\documentclass[11pt]{article}
\usepackage[page,toc,titletoc,title]{appendix}
\usepackage[utf8]{inputenc}
\usepackage{float}
\usepackage{amsmath,amsthm,amssymb,amscd,amstext,amsfonts,mathscinet}
\usepackage{mathtools}
\usepackage{mathrsfs}
\usepackage{thmtools}
\usepackage{latexsym}
\usepackage{verbatim}
\usepackage{framed}
\usepackage{graphicx}
\usepackage{stmaryrd}
\usepackage{enumitem}
\usepackage{fullpage}
\usepackage{bm}
\usepackage{color}
\usepackage{hyperref}
\usepackage{url}
\usepackage{physics}
\usepackage{dsfont}

\usepackage[nameinlink, noabbrev, capitalize]{cleveref}
\hypersetup{colorlinks={true},linkcolor={blue},citecolor=blue}

\theoremstyle{plain}

\newtheorem*{theorem*}{Theorem}

\newtheorem{thm}{Theorem}[section]
\newtheorem{theorem}[thm]{Theorem}

\newtheorem{lemma}[thm]{Lemma}

\newtheorem*{proposition*}{Proposition}
\newtheorem{conjecture}[thm]{Conjecture}

\newtheorem{example}[thm]{Example}
\newtheorem{claim}[thm]{Claim}
\newtheorem*{claim*}{Claim}

\theoremstyle{remark}

\theoremstyle{remark}

\newcommand{\N}{\mathbb{N}}

\newcommand{\cF}{\mathcal F}
\newcommand{\cG}{\mathcal G}

\newcommand{\cM}{\mathcal M}

\newcommand{\cT}{\mathcal T}

\DeclareMathAlphabet{\mathpzc}{OT1}{pzc}{m}{it}

\def\1{\mathbf{1}} 
\def\0{\mathbf{0}}

\let\latexchi\chi
\makeatletter
\renewcommand\chi{\@ifnextchar_\sub@chi\latexchi}
\newcommand{\sub@chi}[2]{
  \@ifnextchar^{\subsup@chi{#2}}{\latexchi^{}_{#2}}%
}
\newcommand{\subsup@chi}[3]{
  \latexchi_{#1}^{#3}%
}
\makeatother
\usepackage[a4paper, margin=1.0in]{geometry}
\usepackage{bbm}
\usepackage{todonotes}
\DeclareMathOperator{\cl}{cl}

\def\1{\mathbf{1}}                       

\date{}

\begin{document} 

\title{The Implicit Graph Conjecture is False}

\author{ Hamed Hatami \thanks{School of Computer Science, McGill University. \texttt{hatami@cs.mcgill.ca}. Supported by an NSERC grant.}
\and Pooya Hatami \thanks{Department of Computer Science and Engineering, The Ohio State University. \texttt{pooyahat@gmail.com}. Supported by NSF grant CCF-1947546}}


\newcommand{\hamed}[1]{\todo[size=\tiny, color=lightgray]{hh: #1}}
\newcommand{\pooya}[1]{\todo[size=\tiny, color=yellow]{ph: #1}}
\newcommand{\blue}[1]{\textcolor{blue}{#1}}

\maketitle

\begin{abstract}
An efficient implicit representation of an $n$-vertex graph $G$ in a family $\cF$ of graphs  assigns to each vertex  of $G$ a binary code of length $O(\log n)$ so that the adjacency between every pair of vertices can be determined only as a function of their codes. This function can depend on the family but not on the individual graph. Every family of graphs admitting such a representation contains at most $2^{O(n\log(n))}$ graphs on $n$ vertices, and thus has at most factorial speed of growth.  

The Implicit Graph Conjecture states that, conversely, every hereditary graph family with at most factorial speed of growth admits an efficient implicit representation. We refute this conjecture by establishing the existence of hereditary graph families with factorial speed of growth that require codes of length $n^{\Omega(1)}$. 
\end{abstract}


\section{Introduction}

In the context of efficient representations of graphs, the concept of implicit representations  was formally defined by M\"uller \cite{Muller}~and Kannan, Naor, and Rudich~\cite{KannanSTOC,MR1186827}. Let $\cF$ be a family of (labeled) finite graphs, and let $\cF_n$ denote the $n$-vertex graphs in $\cF$ with vertex set $[n]\coloneqq\{1,\ldots,n\}$. For a function $w:\N \rightarrow \N$, an \emph{implicit representation} with code-length $w(n)$ for an $n$-vertex graph $G$ in $\cF$ is an assignment of a $w(n)$-bit code to each vertex of $G$, together with a \emph{decoder} function that takes  two vertex-codes and  returns whether or not they are adjacent in $G$. Here, the decoder function may depend on $\cF$, but it is independent of the individual graph $G$.

Note that even the family of all graphs can be represented using codes of length $n$, since every vertex can be labeled with the corresponding row of the adjacency matrix.  We call an implicit representation \emph{efficient} if its code-length is $w(n)=O(\log n)$.  The literature often refers to efficient implicit representations  simply as implicit representations. However, in this article, we are interested in obtaining lower bounds on the code-length, and as a result, we have chosen to allow implicit representations to have arbitrary code-lengths. 

Moreover, the original definitions of \cite{Muller,MR1186827} require that the decoder function is computable by a polynomial-time algorithm. However,  as these computational matters are irrelevant to our lower bound, we do not require them here.

\begin{example}\label{ex:forests}
Let $\cT$ be the family of all forests. Given any $G\in \cT_n$, we can turn each connected component of $G$ to a rooted tree by choosing an arbitrary root, and assign to every vertex of $G$, the index of its parent (if it exists). Note that the adjacency of a pair of vertices can be decided given their $\lceil \log n\rceil$-length codes, and thus $\cT$ has an efficient implicit representation.
\end{example}

A sequence $(U_n)_{n\in \N}$ of graphs is said to be a sequence of \emph{universal graphs} for a graph family $\cF$, if for every $n$ and $F\in \cF_n$, there exists a graph embedding $\pi_F:V(F)\to V(U_n)$. The family $\cF$ is said to be representable by polynomial-size \emph{universal graphs} if these universal graphs satisfy  $|V(U_n)|=n^{O(1)}$. Universal graphs were initially defined with the purpose of representation of the family of all graphs~\cite{MR172268, MR177911}, and were extensively studied for various families of graphs~\cite{MR505812, MR556017, MR806964, wu1985embedding, MR990447, Muller,MR1186827, MR611926, MR3899158, bonamy2021optimal}. Efficient implicit representations are equivalent to polynomial-size universal graphs~\cite{Muller,MR1186827}. To see the equivalence, note that given an efficient implicit representation, one can take $V(U_n)$ to be the set of  all $O(\log(n))$-bits codes and define the edges of $U_n$ according to the decoder function. For the converse direction, one can assign a unique $O(\log(n))$-bit code to each of the $n^{O(1)}$ vertices of $U_n$, and use the functions $\pi_F$ to assign the codes of the vertices of each $F \in \cF_n$ accordingly. 

A simple counting argument shows that families of graphs with efficient implicit representations cannot be very large. Indeed, if $\cF$ has an efficient implicit representation, then every graph in $\cF_n$ can be described using $O(n \log n)$ bits, and thus we must have $|\cF_n|\leq 2^{O(n \log n)}$. In the terminology of~\cite{MR1490438,MR1769217} such families are said to have at most \emph{factorial} speed of growth. 

A graph family $\cF$ is called \emph{hereditary} if it is closed under taking \emph{induced} subgraphs. More precisely, for every $m \le n$,  injection $\sigma:[m]\to [n]$, and  $G \in \cF_n$, the graph $F$ with the vertex set $[m]$ and the edge set $\{ij:\sigma(i)\sigma(j) \in E(G)\}$  belongs to $\cF$.  Note that taking $m=n$ in this definition  ensures that $\cF$ is closed under graph isomorphism. The  \emph{hereditary closure} of a family $\cF$, denoted by $\cl(\cF)$, is the set of all induced subraphs of every $G \in \cF$. It is the smallest hereditary family that contains $\cF$.

Next, we discuss why in the study of implicit representations, it is more natural to consider hereditary graph families. Let $\cF$ be  representable by the sequence of polynomial-size universal graphs $U_n$. Let $\epsilon>0$ be a fixed constant, and let $U'_n$ be the disjoint union  $U'_n=U_1 \cup \ldots \cup U_n$. Let $\cF_\epsilon \subseteq \cl(\cF)$ be obtained from $\cF$ by including all induced subgraphs of size at least $n^\epsilon$ of every $G \in \cF_n$ for every $n$. Since the size of $U'_n$ is polynomial in $n^\epsilon$, the family $\cF_\epsilon$ is representable by the universal graphs $U'_n$. This observation shows that for $\cF$ to have an efficient implicit representation, not only $\cF$ must have at most factorial speed of growth, but more generally, the families $\cF_\epsilon$ must also have at most factorial speed of growth.  

\begin{example}(See also~\cite{MR1971502})
Let $\cF$ be the set of all graphs with at least $n - \sqrt{n}$ isolated vertices. Note that the speed of $\cF$ is at most ${n \choose \sqrt{n}}2^{{\sqrt{n} \choose 2}} =2^{O(n)}$, which is subfactorial. On the other hand, $\cF_{1/2}$ contains all graphs and thus has the super-factorial growth speed of $2^{\Theta(n^2)}$. Hence $\cF$ does not admit an efficient implicit representation despite having at most factorial growth speed.  
\end{example}

Note that $\cF_\epsilon \subseteq \cl(\cF)$.  
In light of these observations, the \emph{implicit graph conjecture} speculates that if $\cl(\cF)$  has at most factorial speed of growth, then $\cF$ admits an efficient implicit representation. 

\begin{conjecture}[Implicit Graph Conjecture \cite{KannanSTOC,MR1186827}, \cite{MR1971502}]
\label{conj:IGC}
Every hereditary graph family of at most factorial speed has an efficient implicit representation. 
\end{conjecture}

\cref{conj:IGC} was posed as a question in~\cite{KannanSTOC,MR1186827}, but later, Spinrad stated it as a  conjecture in \cite{MR1971502}. The implicit graph conjecture is important because it would imply the existence of space-efficient representations for all small hereditary graph families. Moreover, assuming that the decoder is efficiently computable, answering edge queries would require only poly-logarithmic time in the size of the graph.

The implicit graph conjecture has seen much attention by both combinatorists and computer scientists~\cite{MR1186827, MR1971502, MR3613451, MR3490233, MR3899158}, and it has been verified for numerous  restricted classes of hereditary graph families~\cite{MR3702457, MR1696252, alstrup2002small, gavoille2007shorter, MR3279268, MR4232068, MR4262551, harms2021randomized,MR2995722}. 

In this paper, we disprove this three-decades-old conjecture via a  probabilistic argument. 

\begin{theorem}[Main Theorem]
\label{thm:main}
For every $\delta>0$,  there is a hereditary graph family $\cF$ that has at most factorial speed of growth, but every implicit representation of $\cF$ must use codes of length $\Omega(n^{\frac{1}{2}-\delta})$.
\end{theorem}

\cref{thm:main} not only refutes the existence of $O(\log(n))$-bit codes, but it shows that there are factorial hereditary graph families that require  code-lengths that are  polynomial in $n$. 

\paragraph{Proof overview.} First, note that the number of polynomial-size universal graphs is  $2^{n^{O(1)}}$, and  each such graph can represent at most $n^{O(n)}$ graphs of size $n$.  Since there are $2^{\Omega(n^2)}$ graphs on $n$ vertices, a simple counting argument shows that, with high probability,   a random  collection of $n^{\omega(1)}$  graphs on $n$ vertices cannot  be represented by any universal graph of size $n^{O(1)}$.  It is thus tempting to construct $\cF$ by randomly selecting $n^{\omega(1)}$ graphs for each $n$. Then, as desired, with high probability, $\cl(\cF)$  does not have a polynomial-size universal graph representation, but unfortunately, taking the hereditary closure  will expand the family to contain all graphs and thus $\cl(\cF)$ will have growth speed $2^{\Omega(n^2)}$. Indeed for every $n$, for sufficiently large $m$,  even a single random graph on $m$ vertices will, with high probability, contain all graphs on $n$ vertices as induced subgraphs.  

The key idea to overcome this issue is to consider  slightly sparser random graphs. We will select the initial graphs $\cF$ from the set of graphs that have $n^{2-\epsilon}$ number of edges for some $\epsilon>0$. There are still many such graphs; thus, the same counting argument shows that, with high probability, polynomial-size universal graphs cannot represent this family.  Regarding the growth speed of  $\cl(\cF)$, note that every $n$-vertex graph $F \in \cl(\cF)$ is an induced subgraph of some randomly selected  graph $G \in \cF_m$ where $m \ge n$.  If $m$ is much larger than $n$, then $F$ is a small subgraph of the random sparse graph $G$, and thus it has a simple structure. We will show that due to these structural properties, the number of such graphs is at most factorial.  On the other hand, if $m$ is not much larger than $n$,  there are only $n^m$ possible ways of choosing $F$ inside $G$, and thus the number of such graphs $F$ can  be upper-bounded  by the sum of $n^m|\cF_m|$ over $m=n^{O(1)}$.
Since also  $|\cF_m|$ is not very large as a function of $m$,  this sum will be at most $2^{O(n \log(n))}$.

\paragraph{Acknowledgement} 
We wish to thank Nathaniel Harms and Viktor Zamaraev for their feedback on an earlier draft of this paper. 

\section{Preliminary lemmas}

Denote by $B(n,m)$ for $0\leq m\leq n^2$, the set of all bipartite graphs with equal parts $\{1,\ldots, n\}$ and $\{n+1,\ldots, 2n\}$, and exactly $m$ edges.  Recall that a graph is said to be $k$-\emph{degenerate} if all its induced subgraphs have a vertex of  degree at most $k$.

\begin{lemma}
\label{lem:lemDeg}
Suppose $0\leq \epsilon'<\epsilon<1$ are fixed constants. Consider the random graph $G$ chosen uniformly at random from $B(n,m)$ with $m=\lfloor n^{2-\epsilon}\rfloor$. Then, for sufficiently large $n$, with probability $1-o(1)$, every induced subgraph of $G$ with at most $n^{\epsilon'}$ vertices is $(c-1)$-degenerate where $c=\frac{4}{\epsilon-\epsilon'}$. 
\end{lemma}
\begin{proof}
Let $F$ be an induced subgraph of $G$ with at most $n^{\epsilon'}$  vertices. Let $A$ and $B$ denote the set of the vertices of $F$ in the first part and the second part, respectively, and denote $a=|A|$ and $b=|B|$. We have $a,b\leq n^{\epsilon'}$.  We call $F$ \emph{bad} if its minimum degree  is at least $c$.

The probability that every vertex in $A$ has degree at least $c$ is bounded from above by
$$\frac{{b \choose c}^a {n^2-ac \choose m-ac}}{{n^2 \choose m}} \le {b \choose c}^a \left(\frac{m}{n^2}\right)^{ac}\le \left(\frac{mb}{n^2}\right)^{ac} =  \left(\frac{b}{n^{\epsilon}}\right)^{ac}. $$
Similarly the probability that every vertex in $B$ has degree at least $c$ is at most $\left(\frac{a}{n^{\epsilon}}\right)^{bc}$. Thus, by a union bound, the probability that a bad $F$ with parts of size $a$ and $b$, with $a<b\le n^{\epsilon'}$, exists is at most
$$
{n \choose a}{n \choose b} \left(\frac{a}{n^{\epsilon}}\right)^{bc} \leq n^{2b} \left(\frac{a}{n^{\epsilon}}\right)^{bc}  =  \left(n^2 \cdot \frac{a^c}{n^{c\epsilon}}\right)^{b}
\leq \left(\frac{n^2}{n^{c(\epsilon-\epsilon')}}\right)^b= o(n^{-2}), 
$$
where we used the assumption that  $c=\frac{4}{\epsilon-\epsilon'}$.  A similar bound holds for the case  $b<a\le n^{\epsilon'}$. Thus, by a union bound, since there are at most $n^2$ choices for $a$ and $b$, the probability that $G$  has a bad induced subgraph  with parts of size $a,b\leq n^{\epsilon'}$ is $o(1)$.  

It follows that with probability $1-o(1)$, every induced subgraph of $G$ with at most $n^{\epsilon'}$ vertices contains a vertex of degree less than $c$. Consequently, every induced subgraph of $G$ with $n^{\epsilon'}$ many vertices is $(c-1)$-degenerate. 
\end{proof}

\begin{lemma}\label{lemma:degencount}
The number of $c$-degenerate $n$-vertex graphs is at most $2^{O(cn\log n)}$. 
\end{lemma}
\begin{proof}
By deleting the vertices of degree at most $c$ in turn, one obtains an ordering of the vertices   such that every vertex has at most $c$ neighbours in the subsequent vertices. Hence, the number of such graphs is at most
$$O(n! (n^c)^n)=2^{O(cn \log n)}. $$
\end{proof}

\section{Existence of counter-examples} 
We are now ready to prove \cref{thm:main}. Call an $n$-vertex bipartite graph \emph{good} if it is a positive instance of \cref{lem:lemDeg} for  $\epsilon=\frac{1}{2}+\frac{\delta}{2}$, and $\epsilon'=\frac{1}{2}$. Let $\cG$ be the family of all good graphs. Note that all good graphs are bipartite graphs with even number of vertices, and bipartition $\{1,\ldots, \frac{n}{2}\}$ and $\{\frac{n}{2}+1,\ldots, n\}$. We consider hereditary families that are constructed by picking small subsets $\cM_n\subseteq \cG_n$ and taking the hereditary closure of $\cM=\cup_{n\in \N} \cM_n$. To this end, we first prove that as long as $\cM_n$ are not too large, the resulting hereditary family will have at most factorial speed. 

\begin{claim}\label{claim:factorialspeed}
For every $n$, let $\cM_n\subseteq \cG_n$ be a subset with $|\cM_n|\leq 2^{\sqrt{n}}$. The hereditary closure of $\cup_{n\in \N} \cM_n$ has at most factorial speed.  
\end{claim} 
\begin{proof}
Let $\cF$ denote the hereditary closure of $\cup_n \cM_n$. For an $n$-vertex graph $G\in \cF$, let $m$ be the smallest integer such that $G\in \cl(\cM_m)$. We consider two cases, based on whether $n \leq  \sqrt{m/2}$. If $n \leq \sqrt{m/2}$, then since the graphs in $\cM_m$ are good, $G$ is $c$-degenerate for $c=\frac{4}{\epsilon-\epsilon'} = 8 \delta$. By \cref{lemma:degencount}, the number of graphs $G$ of this type is bounded by $2^{O(n\log n)}$. The number of graphs $G$, with $n>\sqrt{m/2}$ is bounded by
\[
\sum_{m=n}^{2  n^2} m^n |\cM_m|= 
\sum_{m=n}^{2  n^2} m^n 2^{\sqrt{m}}
\le 
2 n^2 \cdot (2 n^2)^n \cdot 2^{\sqrt{2 n^2}} = 2^{O(n\log n)}. \qedhere
\]
\end{proof}

We will focus our attention to even $n$. Define $k_n=\lceil 2^{\sqrt{n}}\rceil$. We will show that there is an $\cM_n$ as in the assumption of \cref{claim:factorialspeed} that does not have a universal graph representation of size $u\leq 2^{n^{\frac{1}{2}-\delta}}$. Combined with \cref{claim:factorialspeed}, the closure of $\cup_{n\in \N} \cM_n$ is a hereditary family of at most factorial speed that does not have a  $2^{n^{\frac{1}{2}-\delta}}$-size universal graph representation. Equivalently there is no implicit representation of this family with codes of length at most $n^{\frac{1}{2}-\delta}$, refuting the Implicit Graph Conjecture. 

Suppose $U$ is a $u$-vertex graph. The number of $n$-vertex graphs that can be represented by $U$ is at most $u^n$. Thus the number of collections of  $k_n$ graphs on $n$-vertices that are simultaneously represented by $U$ is at most $u^{nk_n}$.

Since the number of distinct $u$-vertex graphs $U$ is at most $2^{u^2}$, the number of collections of  $k_n$ graphs on $n$-vertices that have a $u$-vertex universal graph is at most 
\begin{equation}
\label{eq:UniversalPack}
2^{u^2}\cdot u^{nk_n}.
\end{equation}

On the other hand, let us first estimate $|\cG_n|$. The number of graphs in the support of $B(\frac{n}{2}, \lfloor( \frac{n}{2})^{2-\epsilon}\rfloor)$ is 
$$
{\frac{n^2}{4} \choose  \lfloor(\frac{n}{2})^{2-\epsilon}\rfloor } \geq 2^{\Omega(n^{2-\epsilon} \log n)}. 
$$
Hence by \cref{lem:lemDeg}, $|\cG_n|$ is also at least $2^{\Omega(n^{2-\epsilon} \log n)}$. As a result, the number of choices of $\cM_n\subseteq \cG_n$ with $|\cM_n|=k_n$ is at least
$$ 2^{\Omega(k_n  n^{2-\epsilon} \log n)}.
$$

We finally observe that for sufficiently large $n$, this is larger than  (\ref{eq:UniversalPack}):
\begin{align*}
\log (2^{u^2}\cdot u^{nk_n}) &= u^2 + n k_n \log(u)\leq 2^{2n^{\frac{1}{2}-\delta}} + n \lceil2^{\sqrt{n}}\rceil n^{\frac{1}{2}-\delta}=O( 2^{\sqrt{n}} n^{\frac{3}{2}-\delta}),
\end{align*}
while 
\begin{align*} 
\log(2^{\Omega(k_n  n^{2-\epsilon} \log n)}) &=\Omega(k_n n^{2-\epsilon} \log n)=\Omega(2^{\sqrt{n}} n^{2-\epsilon} \log n)=\Omega(2^{\sqrt{n}} n^{\frac{3}{2}-\frac{\delta}{2}} \log n).
\end{align*}

Thus, there exist collections $\cM_n\subseteq \cG_n$ with $|\cM_n|\leq k_n$ such that $\cM_n$ cannot be represented by any universal graph of size $u$. This combined with \cref{claim:factorialspeed} concludes the proof of \cref{thm:main}.

\section{Concluding remarks}
The minimum code-length of an implicit representation for the family of all graphs is  $\frac{n}{2}+O(1)$~\cite{MR3613451, MR177911}.
\cref{thm:main} shows that there are factorial hereditary graph properties that require  code-lengths of roughly $\Omega(\sqrt{n})$. We do not know whether the square root appearing in \cref{thm:main} can be improved to a larger power, or even potentially to powers that are arbitrarily close to $1$. To this end, we ask whether there exists an $\epsilon>0$ such that every factorial hereditary property can be represented by codes of length $O(n^{1-\epsilon})$.

The refutation of the Implicit Graph Conjecture leaves the problem of determining whether there is a simple characterization of hereditary graph properties with efficient implicit representations wide open. A  problem of similar nature in the area of communication complexity is to characterize the hereditary classes of Boolean matrices that have randomized communication complexity $O(1)$. Indeed as it was shown in \cite{harms2021randomized}, a derandomization argument shows that every such class must admit an efficient implicit representation, and thus the two problems are closely related. We refer the reader to \cite{harms2021randomized}  for a thorough discussion of the connections to communication complexity. 

The observation that even polynomially large subgraphs of  the random graph $G(n,n^{2-\epsilon})$ are with high probability  $O(1)$-degenerate was first used in \cite{Hambardzumyan2021DimensionfreeBA} as a proof-barrier for  certain conjectures regarding randomized communication complexity. The connection between randomized communication complexity and the implicit graph conjecture was  pointed out  by~\cite{harms2021randomized}. The present paper combines the above-mentioned observation about sparse random graphs with  an information-theoretic counting argument to create a factorial hereditary property that does not admit an efficient implicit representation.

\bibliographystyle{alpha}
\bibliography{refs}

\end{document}